\documentclass[11pt]{amsart}
\usepackage{amsmath,amssymb,amsthm}
\usepackage[margin=1.1in]{geometry}
\usepackage[colorlinks=true,linkcolor=blue,citecolor=blue,urlcolor=blue]{hyperref}

\newtheorem{theorem}{Theorem}[section]
\newtheorem{proposition}[theorem]{Proposition}
\newtheorem{lemma}[theorem]{Lemma}
\newtheorem{corollary}[theorem]{Corollary}
\newtheorem{conjecture}[theorem]{Conjecture}
\theoremstyle{definition}
\newtheorem{definition}[theorem]{Definition}
\newtheorem{example}[theorem]{Example}
\theoremstyle{remark}
\newtheorem{remark}[theorem]{Remark}

\numberwithin{equation}{section}

\newcommand{\R}{\mathbb{R}}
\newcommand{\N}{\mathbb{N}}
\newcommand{\pa}{\pi_1}
\newcommand{\pb}{\pi_2}
\newcommand{\tr}[1]{|#1|}
\newcommand{\Dl}{\overline{\Delta}}

\begin{document}

\title{On a property of plane curves}

\author{Mohammad Javaheri}
\address{Department of Mathematics and Statistics, Siena University\\
515 Loudon Road, Loudonville, NY 12211}
\email{mjavaheri@siena.edu}

\date{Revised, October 2026}

\begin{abstract}
Let $\gamma:[0,1]\to\R^2$ be a continuous curve with $\gamma(0)=(0,0)$ and
$\gamma(1)=(1,1)$. We prove that if $\gamma(t)\in(0,1)^2$ for all
$t\in(0,1)$, then for each $n\in\N$ there exist points $A_1,\dots,A_n$ on
$\gamma$ such that, with $A_0=(0,0)$ and $A_{n+1}=(1,1)$, the horizontal
increments $\pa(\overrightarrow{A_iA_{i+1}})$ and the vertical increments
$\pb(\overrightarrow{A_iA_{i+1}})$, $0\le i\le n$, are positive and coincide
up to order. More generally, the same conclusion holds as soon as $\gamma$
contains an arc that joins two points of the diagonal segment, lies on one
side of the diagonal, and satisfies a mild one-sided boundary condition. This
covers, for instance, the graph of every continuous function $f$ with
$f(0)=0$ and $f(1)=1$, with no restriction on the range of $f$. In all cases
the rearrangement can be taken to be a cyclic shift by one place of a block of
consecutive increments. The main tool is a version of the mountain climbers'
theorem in which one function is piecewise monotone and the other is an
arbitrary continuous function. As an application, any two non-atomic
probability measures on $[0,1]$ satisfying a mild condition admit, for each
$n$, a partition of $[0,1]$ into $n+1$ consecutive intervals whose two lists
of measures are rearrangements of each other.
\end{abstract}

\maketitle

\section{Introduction}

Throughout, $O=(0,0)$, $D=(1,1)$, and $\pa,\pb:\R^2\to\R$ denote the
projections onto the $x$-axis and the $y$-axis. For a continuous curve
$\gamma:[0,1]\to\R^2$ we write $\tr{\gamma}=\gamma([0,1])$ for its trace, and
``a point on $\gamma$'' means a point of $\tr{\gamma}$. The \emph{diagonal} is
the line $\{x=y\}$, and $[O,D]$ is the closed segment joining $O$ and $D$.

Let $\gamma:[0,1]\to[0,1]^2$ be a continuous curve with $\gamma(0)=O$,
$\gamma(1)=D$ and $\gamma(t)\in(0,1)^2$ for $t\in(0,1)$. Applying the
intermediate value theorem to $\pa\circ\gamma+\pb\circ\gamma-1$, which equals
$-1$ at $t=0$ and $1$ at $t=1$, we see that $\gamma$ meets the anti-diagonal
$\{x+y=1\}$ at some point $A=(a,1-a)$ with $0<a<1$. The vectors
$\overrightarrow{OA}=(a,1-a)$ and $\overrightarrow{AD}=(1-a,a)$ satisfy
\begin{equation}\label{eq:intro}
\pa(\overrightarrow{OA})=\pb(\overrightarrow{AD}),\qquad
\pb(\overrightarrow{OA})=\pa(\overrightarrow{AD}).
\end{equation}
We are interested in generalizing this observation to more than one
intermediate point.

\begin{definition}\label{def:config}
Let $n\ge 0$. An \emph{$n$-configuration} on a curve $\gamma$ from $O$ to $D$
is a sequence of points $A_1,\dots,A_n$ on $\gamma$; we always set $A_0=O$ and
$A_{n+1}=D$, write $A_i=(p_i,q_i)$, and put
\[
\alpha_i=p_{i+1}-p_i,\qquad \beta_i=q_{i+1}-q_i,\qquad 0\le i\le n .
\]
The configuration is
\begin{itemize}
\item \emph{balanced} if all $\alpha_i$ and $\beta_i$ are positive and
$(\beta_0,\dots,\beta_n)$ is a rearrangement of $(\alpha_0,\dots,\alpha_n)$;
\item \emph{cyclic} if $\beta_0=\alpha_n$ and $\beta_i=\alpha_{i-1}$ for
$1\le i\le n$, i.e.\ the vertical increments are the horizontal increments
shifted cyclically by one place.
\end{itemize}
\end{definition}

The observation \eqref{eq:intro} says that $\gamma$ carries a balanced (and
cyclic) $1$-configuration. Our first result extends this to every $n$.

\begin{theorem}\label{thm:main-square}
Let $\gamma:[0,1]\to[0,1]^2$ be a continuous curve such that $\gamma(0)=O$,
$\gamma(1)=D$, and $\gamma(t)\in(0,1)^2$ for $t\in(0,1)$. Then for each
$n\in\N$ there exists a balanced $n$-configuration on $\gamma$; that is, there
are points $A_1,\dots,A_n$ on $\gamma$ such that, with $A_0=O$ and
$A_{n+1}=D$, the two sequences
\begin{equation}\label{eq:seqs}
\pa(\overrightarrow{A_iA_{i+1}})>0 \quad\text{and}\quad
\pb(\overrightarrow{A_iA_{i+1}})>0,\qquad i=0,\dots,n,
\end{equation}
are the same after a rearrangement.
\end{theorem}

Theorem~\ref{thm:main-square} is a special case of the following more general
statement, which allows the curve to leave the square, places the conditions
only on one arc of the curve, and identifies the rearrangement.

\begin{theorem}\label{thm:arc}
Let $\gamma:[0,1]\to\R^2$ be a continuous curve with $\gamma(0)=O$ and
$\gamma(1)=D$. Suppose there are $0\le a<b\le 1$ such that
$\gamma(a)=(c,c)$ and $\gamma(b)=(d,d)$ with $0\le c<d\le 1$, that
$\pb\gamma(t)<\pa\gamma(t)$ for all $t\in(a,b)$, and that at least one of
the following holds:
\begin{enumerate}
\item[(i)] $\pa\gamma(t)<d$ for all $t\in(a,b)$;
\item[(ii)] $\pb\gamma(t)>c$ for all $t\in(a,b)$.
\end{enumerate}
Then for each $n\in\N$ there exists a balanced $n$-configuration on $\gamma$
whose points lie on the arc $\gamma([a,b])$. Moreover, the rearrangement can
be chosen to be a cyclic shift by one place of a block of consecutive
increments, all other increments satisfying $\alpha_i=\beta_i$.
The same holds if $\pa$ and $\pb$ are interchanged throughout (so that the arc
lies strictly above the diagonal).
\end{theorem}

\begin{corollary}\label{cor:cases}
Let $\gamma:[0,1]\to\R^2$ be continuous with $\gamma(0)=O$ and $\gamma(1)=D$.
Each of the following conditions implies that $\gamma$ carries a balanced
$n$-configuration for every $n\in\N$.
\begin{enumerate}
\item[(a)] $\gamma(t)=(t,f(t))$ or $\gamma(t)=(f(t),t)$ for a continuous
function $f:[0,1]\to\R$ with $f(0)=0$ and $f(1)=1$.
\item[(b)] $\gamma(t)\in(-\infty,1)^2$ for all $t\in[0,1)$, and every point of
$\gamma$ on the diagonal lies in $[O,D]$.
\item[(c)] $\gamma(t)\in(0,\infty)^2$ for all $t\in(0,1]$, and every point of
$\gamma$ on the diagonal lies in $[O,D]$.
\end{enumerate}
\end{corollary}

Note that in~(a) the function $f$ may take arbitrary real values, so the
graph may leave the unit square through the top or the bottom. Some condition
is necessary: Example~\ref{ex:L} below describes a curve in $[0,1]^2$ which
carries no balanced $n$-configuration for any $n$; it violates both (i) and
(ii) of Theorem~\ref{thm:arc}.

The key tool is the following version of the mountain climbers' theorem,
proved in \S\ref{sec:mc}; here \emph{piecewise monotone} means strictly
monotone on each interval of some finite partition of $[0,1]$.

\begin{theorem}\label{thm:mc-intro}
Let $f_1,f_2:[0,1]\to[0,1]$ be continuous with $f_1(0)=f_2(0)=0$ and
$f_1(1)=f_2(1)=1$. If $f_1$ is piecewise monotone, then there are continuous
$g_1,g_2:[0,1]\to[0,1]$ with $g_j(0)=0$, $g_j(1)=1$ and
$f_1\circ g_1=f_2\circ g_2$.
\end{theorem}

\emph{Outline.} In \S\ref{sec:shift} we show that cyclic configurations are
characterized by a simple system of equations, and that on a curve lying
below the diagonal every cyclic configuration is automatically balanced;
we also give a direct elementary proof for graphs. Section~\ref{sec:mc}
contains Theorem~\ref{thm:mc-intro}. In \S\ref{sec:existence} we construct
cyclic configurations on curves in the square, first when $\pb\circ\gamma$ is
piecewise monotone and then in general by approximation.
Section~\ref{sec:proofs} contains the proofs of
Theorems~\ref{thm:main-square},~\ref{thm:arc} and
Corollary~\ref{cor:cases}. Section~\ref{sec:measures} gives an application to
partitions of an interval, and \S\ref{sec:remarks} discusses sharpness and open
problems.

\section{Cyclic configurations}\label{sec:shift}

\begin{lemma}\label{lem:shift}
Let $A_0=O,A_1,\dots,A_n,A_{n+1}=D$ be points in $\R^2$, $n\ge1$, with
notation as in Definition~\ref{def:config}, and put $s=1-p_n$. Then the
sequence is cyclic if and only if
\begin{equation}\label{eq:shift}
q_i=p_{i-1}+s,\qquad 1\le i\le n .
\end{equation}
In that case
\begin{equation}\label{eq:alpha}
\alpha_i=(p_{i+1}-q_{i+1})+s,\qquad 0\le i\le n .
\end{equation}
\end{lemma}

\begin{proof}
If the sequence is cyclic then, for $1\le i\le n$,
\[
q_i=\sum_{j=0}^{i-1}\beta_j=\alpha_n+\sum_{j=1}^{i-1}\alpha_{j-1}
=(1-p_n)+p_{i-1}=p_{i-1}+s .
\]
Conversely, assume \eqref{eq:shift}. Then $\beta_0=q_1=p_0+s=s=\alpha_n$;
for $1\le i\le n-1$ we have $\beta_i=q_{i+1}-q_i=p_i-p_{i-1}=\alpha_{i-1}$;
and $\beta_n=1-q_n=1-p_{n-1}-s=p_n-p_{n-1}=\alpha_{n-1}$.
Finally, for $0\le i\le n-1$, \eqref{eq:shift} gives
$\alpha_i=p_{i+1}-p_i=p_{i+1}-q_{i+1}+s$, while $\alpha_n=s$ and
$p_{n+1}-q_{n+1}=0$.
\end{proof}

Let $\Delta=\{(x,y):y<x\}$ and $\Dl=\{(x,y):y\le x\}$.

\begin{lemma}\label{lem:positive}
Let $K\subset\R^2$ satisfy $K\subset\{(x,y):x\le 1\}$ and
$K\cap\{x=1\}\subset\{D\}$, and let $A_1,\dots,A_n\in K$ form a cyclic
sequence as in Lemma~\ref{lem:shift}. Then $s>0$. If moreover $K\subset\Dl$,
then $\alpha_i\ge s>0$ and $\beta_i>0$ for all $i$; hence the sequences
$(p_i)$ and $(q_i)$ are strictly increasing, the points $A_0,\dots,A_{n+1}$
are distinct, $A_1,\dots,A_n\in(0,1)^2$, and the configuration is balanced.
\end{lemma}

\begin{proof}
Since $p_n\le1$ we have $s\ge0$. Suppose $s=0$. If $p_i=1$ for some
$1\le i\le n$, then $A_i\in K\cap\{x=1\}$, so $A_i=D$ and, by
\eqref{eq:shift}, $p_{i-1}=q_i-s=1$. Starting from $p_n=1-s=1$ and iterating
we get $p_0=1$, contradicting $p_0=0$. Hence $s>0$.
If $K\subset\Dl$ then $p_{i+1}-q_{i+1}\ge0$ for all $i$, and
\eqref{eq:alpha} yields $\alpha_i\ge s>0$. The $\beta_i$ are a rearrangement
of the $\alpha_i$, so they are positive too. The remaining assertions follow
since $p_0=q_0=0$ and $p_{n+1}=q_{n+1}=1$.
\end{proof}

\subsection*{The graph case: an elementary proof}
For graphs, cyclic configurations can be written down by a simple recursion.

\begin{lemma}\label{lem:graph}
Let $f:[0,1]\to\R$ be continuous with $f(0)=0$, $f(1)=1$ and $f(x)\le x$
for all $x\in[0,1]$. Then for every $n\in\N$ there is a strictly increasing
sequence $0=x_0<x_1<\dots<x_n<x_{n+1}=1$ such that the points
$A_i=(x_i,f(x_i))$ form a balanced cyclic $n$-configuration on the graph of
$f$. Explicitly,
\begin{equation}\label{eq:graph}
f(x_{i+1})-f(x_i)=x_i-x_{i-1},\qquad 0\le i\le n,
\end{equation}
where $x_{-1}=x_n-1$.
\end{lemma}

\begin{proof}
Define $g_0(x)=x$ and, recursively, $g_{k+1}(x)=x-1+f(g_k(x))$ for $k\ge0$,
wherever the right-hand side is defined. We construct numbers
$0<a_1<a_2<\cdots<1$ such that, for each $k\ge1$, $g_k$ is defined and
continuous on $[a_{k-1},1]$ (with $a_0=0$), $g_k(a_k)=0$, $g_k(1)=1$, and
$0<g_k\le1$ on $(a_k,1]$.

Since $f(z)\le z\le1$ for $z\in[0,1]$, whenever $g_k(x)\in[0,1]$ we have
$g_{k+1}(x)\le x-1+g_k(x)\le x\le1$. Now $g_1(x)=x-1+f(x)$ is continuous on
$[0,1]$ with $g_1(0)=-1$ and $g_1(1)=1$; let $a_1\in(0,1)$ be its largest
zero. Then $0<g_1\le1$ on $(a_1,1]$. Suppose $a_k$ has been constructed. Then
$g_k([a_k,1])\subset[0,1]$, so $g_{k+1}$ is defined and continuous on
$[a_k,1]$, with $g_{k+1}(a_k)=a_k-1+f(0)=a_k-1<0$ and
$g_{k+1}(1)=f(g_k(1))=f(1)=1$. Let $a_{k+1}\in(a_k,1)$ be the largest zero of
$g_{k+1}$ in $[a_k,1]$; then $0<g_{k+1}\le1$ on $(a_{k+1},1]$.

Fix $n$, put $s=1-a_n>0$ and $x_i=g_{n-i}(a_n)$ for $0\le i\le n$, and
$x_{n+1}=1$. Then $x_n=a_n$, $x_0=g_n(a_n)=0$, and $x_i\in(0,1]$ for
$1\le i\le n$ because $a_n>a_{n-i}$. By the definition of $g_k$,
\[
f(x_{i})=x_{i-1}+s,\qquad 1\le i\le n,
\]
which is \eqref{eq:shift} for $A_i=(x_i,f(x_i))$; thus the configuration is
cyclic, and \eqref{eq:graph} is a restatement of this. The graph of $f$ lies in
$\Dl\cap\{x\le1\}$ and meets $\{x=1\}$ only at $D$, so the configuration is
balanced by Lemma~\ref{lem:positive}. Finally $x_i=f(x_{i+1})-s<f(x_{i+1})\le
x_{i+1}$ for $0\le i\le n-1$, so the $x_i$ are strictly increasing.
\end{proof}

\begin{corollary}\label{cor:graph}
Let $f:[0,1]\to\R$ be continuous with $f(0)=0$ and $f(1)=1$. Then for every
$n\in\N$ the graph of $f$ carries a balanced $n$-configuration.
\end{corollary}

\begin{proof}
Let $\Phi=\{x\in[0,1):f(x)=x\}$. If $\Phi$ contains $n$ points in $(0,1)$, the
corresponding points on the diagonal form a balanced configuration with
$\alpha_i=\beta_i$. Otherwise $\Phi$ is finite; let $c=\max\Phi\in[0,1)$ and
$f_c(x)=(f(c+x(1-c))-c)/(1-c)$, $x\in[0,1]$. Then $f_c(0)=0$, $f_c(1)=1$ and
$f_c(x)\ne x$ for $x\in(0,1)$, so either $f_c(x)<x$ on $(0,1)$ or
$f_c(x)>x$ on $(0,1)$. In the second case the function
$\tilde f(x)=1-f_c(1-x)$ satisfies $\tilde f(x)<x$ on $(0,1)$, and the map
$P\mapsto D-P$ carries the graph of $\tilde f$ onto the graph of $f_c$; it
maps a balanced configuration $\tilde A_0,\dots,\tilde A_{m+1}$ to the balanced
configuration $A_i=D-\tilde A_{m+1-i}$ (the increments are the same, in
reverse order). Thus in either case Lemma~\ref{lem:graph} gives a balanced
$m$-configuration on the graph of $f_c$ for every $m$. If $c=0$ we are done.
If $c>0$, take $A_1=(c,c)$ and, for $n\ge2$, the image of a balanced
$(n-1)$-configuration on the graph of $f_c$ under $P\mapsto(c,c)+(1-c)P$.
\end{proof}

Corollary~\ref{cor:graph} is also a special case of Theorem~\ref{thm:arc}; see
\S\ref{sec:proofs}.

\section{The mountain climbers' problem}\label{sec:mc}

The mountain climbers' problem asks whether two climbers on the two sides of
a mountain can travel to the summit while staying at the same altitude at all
times; see \cite{Whittaker,Huneke,Keleti,JL} and the references there.
A function is \emph{locally non-constant} if it is not constant on any
nondegenerate interval.

\begin{theorem}[\cite{Huneke}; see also \cite{Keleti}]\label{thm:huneke}
Let $f_1,f_2:[0,1]\to[0,1]$ be continuous and locally non-constant, with
$f_1(0)=f_2(0)=0$ and $f_1(1)=f_2(1)=1$. Then there are continuous
$g_1,g_2:[0,1]\to[0,1]$ with $g_j(0)=0$, $g_j(1)=1$ and
$f_1\circ g_1=f_2\circ g_2$.
\end{theorem}

We need a version in which one function may be constant on intervals,
provided the other is piecewise monotone. A continuous $f:[0,1]\to\R$ is
\emph{piecewise monotone} if there is a partition $0=z_0<z_1<\dots<z_K=1$
such that $f$ is strictly monotone on each $[z_{j-1},z_j]$. Such an $f$ is
locally non-constant, each level set $f^{-1}(c)$ is finite, every interior
local extremum of $f$ is one of the $z_j$, and the set
$E=\{f(z_0),\dots,f(z_K)\}$ of turning values is finite.

\begin{theorem}\label{thm:mc}
Let $f_1,f_2:[0,1]\to[0,1]$ be continuous with $f_1(0)=f_2(0)=0$ and
$f_1(1)=f_2(1)=1$. If $f_1$ is piecewise monotone, then there exist continuous
$g_1,g_2:[0,1]\to[0,1]$ with $g_j(0)=0$, $g_j(1)=1$ and
$f_1\circ g_1=f_2\circ g_2$.
\end{theorem}

\begin{proof}
If $f_2$ is locally non-constant, this is Theorem~\ref{thm:huneke}. Otherwise
let $\{(a_\alpha,b_\alpha)\}_{\alpha\in J}$ be the connected components of
the union of all open subintervals of $(0,1)$ on which $f_2$ is constant. The
index set $J$ is countable, and $f_2\equiv c_\alpha$ on
$I_\alpha=[a_\alpha,b_\alpha]$ by continuity. The intervals $I_\alpha$ are
pairwise disjoint: if $b_\alpha=a_{\alpha'}$, then $f_2$ would be constant on
$(a_\alpha,b_{\alpha'})$, contradicting maximality.

\emph{Step 1: a locally non-constant modification of $f_2$.} Let $E$ be the
set of turning values of $f_1$. For $\alpha\in J$ let
$d_\alpha=\min\bigl(\{|e-c_\alpha|:e\in E,\ e\ne c_\alpha\}\cup\{1\}\bigr)>0$,
and choose
\[
0<\delta_\alpha\le\min\{d_\alpha/2,\ b_\alpha-a_\alpha\},\qquad
\text{and also }\delta_\alpha\le 1-c_\alpha\text{ if }c_\alpha<1 .
\]
Put $\varepsilon_\alpha=1$ if $c_\alpha<1$ and $\varepsilon_\alpha=-1$ if
$c_\alpha=1$, and
$\varphi_\alpha(x)=4(x-a_\alpha)(b_\alpha-x)/(b_\alpha-a_\alpha)^2$, so that
$\varphi_\alpha$ is strictly concave on $I_\alpha$, vanishes at its endpoints,
and has maximum $1$. Define
\[
f_3(x)=\begin{cases}
c_\alpha+\varepsilon_\alpha\delta_\alpha\varphi_\alpha(x), & x\in I_\alpha,\ \alpha\in J,\\
f_2(x), & x\notin\bigcup_\alpha I_\alpha .
\end{cases}
\]
Then $f_3([0,1])\subset[0,1]$, $f_3(0)=0$, $f_3(1)=1$, and
$|f_3-c_\alpha|\le\delta_\alpha<d_\alpha$ on $I_\alpha$.

$f_3$ is continuous: $\psi=f_3-f_2$ vanishes off
$\bigcup_\alpha(a_\alpha,b_\alpha)$ and is continuous on each $I_\alpha$, with
$|\psi|\le\delta_\alpha\le b_\alpha-a_\alpha$ there. Given $\varepsilon>0$,
only finitely many $I_\alpha$ have length $\ge\varepsilon$; near any point
$x_0$, $|\psi|<\varepsilon$ except possibly on these finitely many closed
intervals, on each of which $\psi$ is continuous and vanishes at the
endpoints. Hence $\psi$, and so $f_3$, is continuous.

$f_3$ is locally non-constant: if $f_3$ were constant on an open interval
$U$, then $U$ could not meet any $(a_\alpha,b_\alpha)$, since $\varphi_\alpha$
is not constant on any subinterval; hence $f_3=f_2$ on $U$, and $U$ would be
contained in some $(a_\alpha,b_\alpha)$, a contradiction.

\emph{Step 2.} By Theorem~\ref{thm:huneke} there are continuous
$h,k:[0,1]\to[0,1]$ with $h(0)=k(0)=0$, $h(1)=k(1)=1$ and
$f_1\circ h=f_3\circ k$. For each $\alpha$, the set
$k^{-1}((a_\alpha,b_\alpha))$ is an open subset of $(0,1)$ (since
$k(0)=0$ and $k(1)=1$ do not lie in $(a_\alpha,b_\alpha)$); let $\mathcal J$
be the family of all connected components $(u,v)$ of these sets,
$\alpha\in J$. The members of $\mathcal J$ are pairwise disjoint open
intervals. For $(u,v)\in\mathcal J$ associated with $\alpha$ we have
$k(u),k(v)\in\{a_\alpha,b_\alpha\}$, hence
\[
f_1(h(u))=f_3(k(u))=c_\alpha=f_1(h(v)),\qquad
|f_1(h(t))-c_\alpha|=|f_3(k(t))-c_\alpha|<d_\alpha\quad(u<t<v).
\]

\emph{Claim: $h(u)=h(v)$.} Otherwise, $\Lambda=f_1^{-1}(c_\alpha)$ is a
finite set containing $h(u)\neq h(v)$, so there are consecutive points $p<q$
of $\Lambda$ in the closed interval between $h(u)$ and $h(v)$. Since
$f_1-c_\alpha$ does not vanish on $(p,q)$, the function $|f_1-c_\alpha|$
attains its maximum on $[p,q]$ at some $r\in(p,q)$, which is an interior local
extremum of $f_1$; thus $f_1(r)\in E$ and $f_1(r)\neq c_\alpha$, whence
$|f_1(r)-c_\alpha|\ge d_\alpha$. By the intermediate value theorem,
$r=h(t_0)$ for some $t_0\in[u,v]$, and $t_0\ne u,v$ because
$r\notin\Lambda$. This contradicts the display above.

\emph{Step 3.} Define $g_2=k$ and
\[
g_1(t)=\begin{cases}
h(u)=h(v), & t\in(u,v)\in\mathcal J,\\
h(t), & t\notin\bigcup\mathcal J .
\end{cases}
\]
As in Step~1, $g_1$ is continuous: $g_1-h$ vanishes off $\bigcup\mathcal J$,
and on $(u,v)\in\mathcal J$ it is bounded by the oscillation of $h$ on
$[u,v]$, which is small for all but finitely many members of $\mathcal J$ by
uniform continuity. Clearly $g_1([0,1])\subset[0,1]$, and $g_1(0)=0$,
$g_1(1)=1$ since $0,1\notin\bigcup\mathcal J$. If $t\notin\bigcup\mathcal J$,
then $k(t)\notin\bigcup_\alpha(a_\alpha,b_\alpha)$, so
$f_2(k(t))=f_3(k(t))=f_1(h(t))=f_1(g_1(t))$. If $t\in(u,v)\in\mathcal J$ is
associated with $\alpha$, then $f_2(k(t))=c_\alpha=f_1(h(u))=f_1(g_1(t))$.
Thus $f_1\circ g_1=f_2\circ g_2$.
\end{proof}

\begin{remark}
In Step~1 the direction of each bump plays no role in Steps~2--3; it is chosen
only to keep $f_3$ inside $[0,1]$. In particular no condition relating the
local maxima and minima of $f_1$ on a level set is needed.
\end{remark}

\section{Cyclic configurations on curves in the square}\label{sec:existence}

In this section $\gamma=(g,f):[0,1]\to[0,1]^2$ is continuous with
$\gamma(0)=O$, $\gamma(1)=D$, and
\begin{equation}\label{eq:H}
g(t)=\pa\gamma(t)<1\quad\text{for all }t\in[0,1),
\end{equation}
so that $\tr\gamma$ meets the line $\{x=1\}$ only at $D$.

\begin{proposition}[Partitioning functions]\label{prop:partition}
Assume $\gamma=(g,f):[0,1]\to[0,1]^2$ is continuous, $\gamma(0)=O$,
$\gamma(1)=D$, and $f$ is piecewise monotone. Then for each $n\in\N$ there
exist continuous functions $y,x_1,\dots,x_n:[0,1]\to[0,1]$ such that, with
$x_0\equiv0$,
\begin{enumerate}
\item[(i)] $(x_i(t),x_{i-1}(t)+y(t))\in\tr\gamma$ for all $t\in[0,1]$ and
$1\le i\le n$;
\item[(ii)] $x_i(0)=y(0)=0$ for $1\le i\le n$;
\item[(iii)] $x_n(1)=1$ and $x_{n-1}(1)+y(1)=1$.
\end{enumerate}
\end{proposition}

\begin{proof}
We use induction on $n$. For $n=1$ take $x_1=g$ and $y=f$. Suppose
$u_1,\dots,u_n,v$ satisfy (i)--(iii). The function $u_n+v$ is continuous,
vanishes at $0$, and $u_n(1)+v(1)=1+v(1)\ge1$; let $t_0\in(0,1]$ be the
smallest $t$ with $u_n(t)+v(t)=1$. Set
\[
f_1=f,\qquad f_2(t)=u_n(t_0t)+v(t_0t),\qquad t\in[0,1].
\]
Both map $[0,1]$ into $[0,1]$ (for $f_2$ by the choice of $t_0$), vanish at
$0$ and equal $1$ at $1$, and $f_1$ is piecewise monotone. By
Theorem~\ref{thm:mc} there are $g_1,g_2$ with $f_1\circ g_1=f_2\circ g_2$.
Define
\[
y(t)=v(t_0g_2(t)),\quad x_i(t)=u_i(t_0g_2(t))\ (1\le i\le n),\quad
x_{n+1}(t)=g(g_1(t)).
\]
Conditions (i) for $i\le n$ and (ii) are inherited from $u_i,v$ (note
$g_2(0)=0$ and $g(g_1(0))=g(0)=0$). For $i=n+1$,
\[
(x_{n+1}(t),x_n(t)+y(t))=\bigl(g(g_1(t)),f_2(g_2(t))\bigr)
=\bigl(g(g_1(t)),f(g_1(t))\bigr)=\gamma(g_1(t)).
\]
Finally $x_{n+1}(1)=g(1)=1$ and $x_n(1)+y(1)=u_n(t_0)+v(t_0)=1$.
\end{proof}

We use the following standard crossing lemma, which follows from the Brouwer
fixed point theorem; see \cite{Maehara}.

\begin{lemma}\label{lem:cross}
Let $R=[0,1]\times[v_1,v_2]$. If $\sigma,\tau:[0,1]\to R$ are continuous,
$\pa\sigma(0)=0$, $\pa\sigma(1)=1$, $\pb\tau(0)=v_1$ and $\pb\tau(1)=v_2$,
then $\tr\sigma\cap\tr\tau\ne\emptyset$.
\end{lemma}

\begin{proposition}\label{prop:pm}
Let $\gamma=(g,f):[0,1]\to[0,1]^2$ be continuous with $\gamma(0)=O$,
$\gamma(1)=D$, satisfying \eqref{eq:H}, and with $f$ piecewise monotone. Then
for every $n\in\N$ there is a cyclic $n$-configuration on $\gamma$.
\end{proposition}

\begin{proof}
For $n=1$ take the point of $\gamma$ on the anti-diagonal, as in the
Introduction; by Lemma~\ref{lem:shift}, a point $A_1=(p_1,q_1)$ forms a cyclic
$1$-configuration if and only if $q_1=1-p_1$.

Let $n\ge2$, and let $y,x_1,\dots,x_{n-1}$ be the partitioning functions of
Proposition~\ref{prop:partition} for $n-1$.

\emph{We have $y(1)>0$.} Suppose $y(1)=0$. If $x_i(1)=1$ for some
$1\le i\le n-1$, then by (i) the point $(1,x_{i-1}(1))$ lies on $\gamma$, so
by \eqref{eq:H} it equals $D$ and $x_{i-1}(1)=1$. By (iii),
$x_{n-1}(1)=1$; iterating we obtain $x_0(1)=1$, a contradiction.

Define $\eta(t)=(1-y(t),\,x_{n-1}(t)+y(t))$. Then $\eta$ is continuous, takes
values in $R=[0,1]\times[-1,2]$, and $\eta(0)=(1,0)$,
$\eta(1)=(1-y(1),1+y(1))$. Let $\tau$ be the path in $R$ obtained by
following the segment from $(1,-1)$ to $(1,0)$, then $\eta$, then the
vertical segment from $\eta(1)$ to $(1-y(1),2)$. By
Lemma~\ref{lem:cross} applied to $\sigma=\gamma$ and $\tau$, the traces of
$\gamma$ and $\tau$ meet. The first segment does not meet $\tr\gamma$ by
\eqref{eq:H}, and the last one lies in $\{y>1\}$ because $y(1)>0$. Hence
$\eta(t^*)\in\tr\gamma$ for some $t^*\in[0,1]$. Put
\[
A_i=(x_i(t^*),\,x_{i-1}(t^*)+y(t^*)),\ 1\le i\le n-1,\qquad A_n=\eta(t^*).
\]
These points lie on $\gamma$. With $s=y(t^*)=1-p_n$ we have
$q_i=p_{i-1}+s$ for $1\le i\le n-1$, and $q_n=x_{n-1}(t^*)+y(t^*)=p_{n-1}+s$.
By Lemma~\ref{lem:shift} the configuration is cyclic.
\end{proof}

\begin{theorem}\label{thm:square-general}
Let $\gamma:[0,1]\to[0,1]^2$ be continuous with $\gamma(0)=O$, $\gamma(1)=D$,
satisfying \eqref{eq:H}. Then for every $n\in\N$ there is a cyclic
$n$-configuration on $\gamma$, and every such configuration has $s>0$, i.e.\
$\alpha_n=\beta_0>0$. If in addition $\tr\gamma\subset\Dl$, every cyclic
configuration on $\gamma$ is balanced, with $(p_i)$ and $(q_i)$ strictly
increasing.
\end{theorem}

\begin{proof}
Write $\gamma=(g,f)$. For $k\ge1$ let $L_k$ be the piecewise linear function
interpolating $f$ at the nodes $j/k$, $0\le j\le k$, and let
$\sigma_1,\dots,\sigma_k$ be its slopes. Choose $\varepsilon_k\in(0,1/k)$
with $\varepsilon_k+(1-\varepsilon_k)\sigma_j\neq0$ for all $j$ (only finitely
many values are excluded), and put $f_k=(1-\varepsilon_k)L_k+\varepsilon_k x$.
Then $f_k$ is piecewise linear with nonzero slopes, hence piecewise monotone;
$f_k([0,1])\subset[0,1]$, $f_k(0)=0$, $f_k(1)=1$; and
$\|f_k-f\|_\infty\le\|L_k-f\|_\infty+\varepsilon_k\to0$. The curve
$\gamma_k=(g,f_k)$ satisfies the hypotheses of Proposition~\ref{prop:pm}, so
there are parameters $\tau^k_1,\dots,\tau^k_n$ such that the points
$\gamma_k(\tau^k_i)$ form a cyclic $n$-configuration on $\gamma_k$. Passing to a
subsequence, $\tau^k_i\to\tau_i$ for each $i$, and by uniform convergence
$\gamma_k(\tau^k_i)\to\gamma(\tau_i)$. Since the equations \eqref{eq:shift}
are preserved under limits, the points $A_i=\gamma(\tau_i)$ form a cyclic
$n$-configuration on $\gamma$. The remaining assertions follow from
Lemma~\ref{lem:positive} with $K=\tr\gamma$.
\end{proof}

\begin{remark}\label{rem:not-below}
Without the assumption $\tr\gamma\subset\Dl$, the cyclic configurations
provided by Theorem~\ref{thm:square-general} still satisfy
$\alpha_n=\beta_0=s>0$, but by \eqref{eq:alpha} the other increments may be
zero or negative where the configuration crosses the diagonal.
\end{remark}

\section{Proofs of the main results}\label{sec:proofs}

\begin{proposition}\label{prop:onesided}
Let $\gamma:[0,1]\to\R^2$ be continuous with $\gamma(0)=O$, $\gamma(1)=D$ and
$\tr\gamma\subset\Dl$. Suppose that
\begin{enumerate}
\item[(a)] $\pa\gamma(t)<1$ for all $t\in[0,1)$, or
\item[(b)] $\pb\gamma(t)>0$ for all $t\in(0,1]$.
\end{enumerate}
Then for every $n\in\N$ there is a balanced cyclic $n$-configuration on
$\gamma$, and its coordinates $(p_i)$ and $(q_i)$ are strictly increasing.
\end{proposition}

\begin{proof}
Assume (a). The set $\{t\in[0,1]:\pb\gamma(t)\le0\}$ is closed, contains $0$
and not $1$; let $T<1$ be its maximum. Then $\pb\gamma>0$ on $(T,1]$, so
$\pb\gamma(T)=0$ and $\gamma(T)=(a',0)$ with $0\le a'<1$ (as
$\pa\gamma\ge\pb\gamma$ and by (a)). Let $\gamma'$ be the curve that first
traverses the segment from $O$ to $\gamma(T)$ and then follows
$\gamma|_{[T,1]}$. Its trace is
$\tr{\gamma'}=([0,a']\times\{0\})\cup\gamma([T,1])$. On $[T,1]$ we have
$0\le\pb\gamma\le\pa\gamma\le1$, so $\tr{\gamma'}\subset[0,1]^2\cap\Dl$, and
$\tr{\gamma'}$ meets $\{x=1\}$ only at $D$. By
Theorem~\ref{thm:square-general}, $\gamma'$ carries a balanced cyclic
$n$-configuration with $(q_i)$ strictly increasing. Since $q_0=0$, we have
$q_i>0$ for $i\ge1$, so no $A_i$ with $1\le i\le n$ lies on the segment
$[0,a']\times\{0\}$; hence all $A_i$ lie on $\gamma([T,1])\subset\tr\gamma$.

Assume (b). Let $\rho(x,y)=(1-y,1-x)$ be the reflection in the anti-diagonal,
and $\hat\gamma(t)=\rho(\gamma(1-t))$. Then $\hat\gamma(0)=O$,
$\hat\gamma(1)=D$, $\tr{\hat\gamma}\subset\Dl$ (as $\rho$ preserves $\Dl$),
and $\pa\hat\gamma(t)=1-\pb\gamma(1-t)<1$ for $t\in[0,1)$. By case (a),
$\hat\gamma$ carries a balanced cyclic configuration $\hat A_0,\dots,\hat
A_{n+1}$. Put $A_i=\rho(\hat A_{n+1-i})$. A direct computation gives
$\alpha_i=\hat\beta_{n-i}$ and $\beta_i=\hat\alpha_{n-i}$, from which one checks
that $(A_i)$ is again a balanced cyclic configuration, now on $\gamma$.
\end{proof}

\begin{proof}[Proof of Theorem~\ref{thm:arc}]
By interchanging the coordinates if necessary we may assume the arc lies below
the diagonal. Let
\[
\eta(t)=\frac{\gamma(a+t(b-a))-(c,c)}{d-c},\qquad t\in[0,1].
\]
Then $\eta(0)=O$, $\eta(1)=D$, $\tr\eta\subset\Dl$, and condition (i)
(resp.\ (ii)) implies condition (a) (resp.\ (b)) of
Proposition~\ref{prop:onesided} for $\eta$. Hence for every $m\ge1$, $\eta$
carries a balanced cyclic $m$-configuration $P_0,\dots,P_{m+1}$; for $m=0$ we
take the trivial configuration $P_0=O$, $P_1=D$. Let
$Q_j=(c,c)+(d-c)P_j\in\gamma([a,b])$; note $Q_0=\gamma(a)$ and
$Q_{m+1}=\gamma(b)$.

Given $n\ge1$, let $m=n-[c>0]-[d<1]$, where $[\cdot]$ is $1$ if the statement
holds and $0$ otherwise. If $m\ge0$, the sequence
\[
O,\ (Q_0\text{ if }c>0),\ Q_1,\dots,Q_m,\ (Q_{m+1}\text{ if }d<1),\ D
\]
has $n$ interior points on $\gamma([a,b])$. Its increments are: $(c,c)$ if
$c>0$; then the increments of $(Q_j)$, which are $(d-c)$ times those of
$(P_j)$ and hence positive and cyclic; then $(1-d,1-d)$ if $d<1$. This is a
balanced $n$-configuration of the stated form. If $m<0$, then $n=1$, $c>0$
and $d<1$, and $A_1=\gamma(a)=(c,c)$ works.
\end{proof}

\begin{lemma}\label{lem:diag}
If $\gamma$ contains infinitely many points of the open segment
$\{(t,t):0<t<1\}$, then $\gamma$ carries, for every $n$, a balanced
$n$-configuration with $\alpha_i=\beta_i$ for all $i$.
\end{lemma}

\begin{proof}
Choose $n$ such points and order them by their first coordinate.
\end{proof}

\begin{proof}[Proof of Corollary~\ref{cor:cases}]
By Lemma~\ref{lem:diag} we may assume that $\tr\gamma$ contains only finitely
many points of the open diagonal segment.

(b) Let $\Omega=\{t\in[0,1):\pa\gamma(t)=\pb\gamma(t)\}$, which contains $0$.
By hypothesis $\gamma(\Omega)$ is contained in the finite set
$F=\{(c,c)\in\tr\gamma:0\le c<1\}$. If $\sup\Omega=1$, then $D$ would be a
limit of points of $F$, which is impossible. So $T=\sup\Omega<1$ belongs to
$\Omega$, $\gamma(T)=(c,c)$ with $0\le c<1$, and on $(T,1)$ the curve lies
strictly on one side of the diagonal. If it lies below, condition (i) of
Theorem~\ref{thm:arc} holds with $a=T$, $b=1$, $d=1$; if above, the
interchanged condition $\pb\gamma<1$ holds. Apply Theorem~\ref{thm:arc}.

(c) Symmetrically, let $\Omega'=\{t\in(0,1]:\pa\gamma(t)=\pb\gamma(t)\}$ and
$a'=\inf\Omega'$. As $\gamma(\Omega')$ is a finite set not containing $O$, we
have $a'>0$, $a'\in\Omega'$ and $\gamma(a')=(d,d)$ with $0<d\le1$. On $(0,a')$
the curve is strictly on one side, and condition (ii) (or its interchanged
form $\pa\gamma>0$) holds with $a=0$, $c=0$, $b=a'$.

(a) For $\gamma(t)=(t,f(t))$, let $T$ be the largest $t\in[0,1)$ with
$f(t)=t$ and $c=T$. On $(T,1)$ either $f(t)<t$, and (i) holds since
$\pa\gamma(t)=t<1$; or $f(t)>t$, and the interchanged form of (ii) holds since
$\pa\gamma(t)=t>T=c$. The case $\gamma(t)=(f(t),t)$ follows by interchanging
coordinates.
\end{proof}

\begin{proof}[Proof of Theorem~\ref{thm:main-square}]
The hypotheses imply condition (b) of Corollary~\ref{cor:cases}.
\end{proof}

\begin{remark}
The proof shows more about the rearrangement in Theorem~\ref{thm:main-square}.
If $\gamma(t)\in\Delta$ for all $t\in(0,1)$, the configuration can be chosen
cyclic. In general, either the configuration can be chosen with
$\alpha_i=\beta_i$ for all $i$, or with $\alpha_0=\beta_0$ and the remaining
increments cyclically shifted, or cyclic.
\end{remark}

\section{Equitable partitions for two measures}\label{sec:measures}

\begin{corollary}\label{cor:measures}
Let $\mu,\nu$ be Borel probability measures on $[0,1]$ without atoms, and
assume that
\begin{enumerate}
\item[(b$'$)] $\mu([0,t])<1$ and $\nu([0,t])<1$ for all $t\in[0,1)$, or
\item[(c$'$)] $\mu([0,t])>0$ and $\nu([0,t])>0$ for all $t\in(0,1]$.
\end{enumerate}
Then for every $n\in\N$ there exist $0=t_0<t_1<\dots<t_{n+1}=1$ such that the
numbers $\mu([t_i,t_{i+1}])$, $0\le i\le n$, are positive and are a
rearrangement of the numbers $\nu([t_i,t_{i+1}])$, $0\le i\le n$.
\end{corollary}

\begin{proof}
Let $F(t)=\mu([0,t])$ and $G(t)=\nu([0,t])$; they are continuous,
nondecreasing, $F(0)=G(0)=0$ and $F(1)=G(1)=1$. The curve $\gamma=(F,G)$ lies
in $[0,1]^2$, so all its diagonal points are in $[O,D]$, and (b$'$),
(c$'$) are conditions (b), (c) of Corollary~\ref{cor:cases}. Write the
resulting points as $A_i=\gamma(t_i)$. Since $F(t_{i+1})-F(t_i)>0$ and $F$ is
nondecreasing, $t_{i+1}>t_i$; and $\mu([t_i,t_{i+1}])=F(t_{i+1})-F(t_i)$,
$\nu([t_i,t_{i+1}])=G(t_{i+1})-G(t_i)$.
\end{proof}

In particular, if $f,g\ge0$ are integrable on $[0,1]$ with
$\int_0^1f=\int_0^1g=1$ and $\int_0^af,\ \int_0^ag\in(0,1)$ for $a\in(0,1)$,
then for every $n$ there is a partition $0=t_0<\dots<t_{n+1}=1$ such that the
numbers $\int_{t_i}^{t_{i+1}}f>0$ are a rearrangement of the numbers
$\int_{t_i}^{t_{i+1}}g>0$. In the language of cake cutting: two agents with
non-atomic valuations of $[0,1]$ can cut it into $n+1$ consecutive pieces so
that the two lists of values agree up to order. For signed integrable $f,g$ an
analogous statement, with the $t_i$ not necessarily ordered, follows from
Corollary~\ref{cor:cases}(b) or (c).

\section{Sharpness and open problems}\label{sec:remarks}

\begin{example}\label{ex:L}
Let $\gamma(t)=(2t,0)$ for $t\in[0,1/2]$ and $\gamma(t)=(1,2t-1)$ for
$t\in[1/2,1]$. In a balanced configuration all $A_i$, $1\le i\le n$, would
have $q_i>0$, hence lie on the right edge, which forces
$\alpha_1=\dots=\alpha_n=0$. Thus $\gamma$ carries no balanced
$n$-configuration for any $n$. Here the only off-diagonal arc has $a=0$,
$b=1$, $c=0$, $d=1$, and both (i) ($\pa\gamma=1$ on $[1/2,1)$) and (ii)
($\pb\gamma=0$ on $(0,1/2]$) fail. So in Theorem~\ref{thm:arc} the two
conditions cannot both be dropped.
\end{example}

If positivity is not required, the curve of Example~\ref{ex:L} does carry
configurations of distinct points whose horizontal and vertical increments
agree up to order. For even $n$ take $A_j=(j/n,0)$ for $0\le j\le n/2$,
$A_{n/2+j}=(1,\,1/2+(j-1)/n)$ for $1\le j\le n/2$, and $A_{n+1}=D$; the
horizontal increments are $1/n$ ($n/2$ times), $1/2$, and $0$ ($n/2$ times),
and so are the vertical ones. For odd $n$ take $A_j=(2j/(n+1),0)$ for
$0\le j\le (n+1)/2$ and $A_{(n+1)/2+j}=(1,2j/(n+1))$ for
$1\le j\le(n+1)/2$. This suggests the following.

\begin{conjecture}\label{conj:general}
Let $\gamma:[0,1]\to\R^2$ be a continuous curve with $\gamma(0)=O$ and
$\gamma(1)=D$. Then for every $n\in\N$ there are distinct points
$A_1,\dots,A_n$ on $\gamma$, different from $O$ and $D$, such that the
sequences $\pa(\overrightarrow{A_iA_{i+1}})$ and
$\pb(\overrightarrow{A_iA_{i+1}})$, $0\le i\le n$, are the same after a
rearrangement.
\end{conjecture}

Proposition~\ref{prop:onesided} shows that a curve in $\Dl$ satisfying (a) or
(b) carries configurations in which $\alpha_i=\beta_{i+1}$ (indices modulo
$n+1$), i.e.\ the horizontal increments equal the vertical ones shifted by one
place. Interchanging the coordinates turns this into a shift by one place in
the other direction. It is natural to ask whether every cyclic shift occurs.

\begin{conjecture}\label{conj:shift}
Let $\gamma:[0,1]\to\R^2$ be continuous with $\gamma(0)=O$, $\gamma(1)=D$ and
$0<\pb\gamma(t)<\pa\gamma(t)<1$ for all $t\in(0,1)$. Then for every $n\in\N$
and every $k\in\{1,\dots,n\}$ there are distinct points $A_1,\dots,A_n$ on
$\gamma$ such that
\[
\pa(\overrightarrow{A_iA_{i+1}})=\pb(\overrightarrow{A_jA_{j+1}})>0,
\qquad j\equiv i+k \pmod{n+1},\quad 0\le i\le n .
\]
\end{conjecture}

The case $k=1$ of Conjecture~\ref{conj:shift} holds, under weaker hypotheses,
by Proposition~\ref{prop:onesided}. The shift $k=0$ is excluded: it would force
all $A_i$ to lie on the diagonal.

\subsection*{Acknowledgment} I would like to thank the referee for mentioning
some of the references and for many other useful comments. This revised version was prepared by an AI assistant (Claude).

\end{document}